\documentclass[a4paper,reqno]{amsart}
\usepackage{verbatim}  
\usepackage[neverdecrease]{paralist}
\usepackage{footmisc}
\newcommand{\included}{\subseteq}

\newcommand{\nincluded}{\nsubseteq}
\newcommand{\includes}{\supseteq}
\usepackage[all]{xy}
\usepackage{mathrsfs}
\newcommand{\sig}{\mathscr S}
\usepackage{upgreek}
\newtheorem{theorem}{Theorem}

\newcommand{\str}[1]{\mathfrak{#1}}  
\newcommand{\diag}[1]{\operatorname{diag}(#1)}
\usepackage{amssymb}  
\renewcommand{\restriction}{\mathbin{\upharpoonright}}

\newcommand{\Forall}[1]{\forall{#1}\;}
\newcommand{\Exists}[1]{\exists{#1}\;}
\newcommand{\lto}{\to}
\newcommand{\stnd}[1]{\mathbb{#1}}
\newcommand{\Z}{\stnd Z}
\newcommand{\Q}{\stnd Q}
\newcommand{\F}{\stnd F}
\usepackage{bm}

\newcommand{\thy}[1]{\mathrm{#1}}      
\newcommand{\DF}{\thy{DF}}             
\newcommand{\ACFA}{\thy{ACFA}}         
\newcommand{\mDF}{\text{$m$-$\thy{DF}$}} 
\newcommand{\mpDF}{\text{$(m+1)$-$\thy{DF}$}}  
\newcommand{\oDF}{\text{$\upomega$-$\thy{DF}$}} 
\newcommand{\DCF}{\thy{DCF}}           
\newcommand{\mDCF}{\text{$m$-$\thy{DCF}$}}  
\newcommand{\oDCF}{\text{$\upomega$-$\thy{DCF}$}}  
\newcommand{\mpDCF}{\text{$(m+1)$-$\thy{DCF}$}}  

\newcommand{\VSp}[1]{\operatorname{VS}_{#1}} 
\newcommand{\VSpm}[1]{\operatorname{VS}_{#1}^{\mathrm m}} 
\newcommand{\VSpr}[1]{\operatorname{VS}_{#1}^{\mathrm r}} 
\newcommand{\VSpst}[1]{{\VSp{#1}}\!^*}      

\newcommand{\Mod}[2][\included]{\operatorname{Mod}^{#1}(#2)}

\renewcommand{\phi}{\varphi}
\renewcommand{\setminus}{\smallsetminus}
\renewcommand{\leq}{\leqslant}

\renewcommand{\vec}[1]{\bm{#1}}
\newcommand{\alg}{^{\mathrm{alg}}}
\newcommand{\lord}{\vartriangleleft}
\usepackage{stmaryrd}
\newcommand{\lordeq}{\trianglelefteqslant}

\usepackage{url}

\begin{document}
  \title{Chains of Theories and Companionability}
\author[\"O. Kasal]{\"Ozcan Kasal}
\address{Middle East Technical University, Northern Cyprus Campus}
\email{kasal@metu.edu.tr}
\author[D. Pierce]{David Pierce}
\address{Mimar Sinan Fine Arts University, Istanbul}
\email{dpierce@msgsu.edu.tr}
\date{\today}
\subjclass[2010]{03C10, 03C60, 12H05, 13N15}
\begin{abstract}
The theory of fields that are equipped with a countably infinite
family of commuting derivations is not companionable; but if the axiom
is added whereby the characteristic of the fields is zero, then the
resulting theory is companionable.  Each of these two theories is the
union of a chain of companionable theories.  In the case of
characteristic zero, the model-companions of the theories in the chain
form another chain, whose union is therefore the model-companion of
the union of the original chain.  However, in a signature with
predicates, in all finite numbers of arguments, for linear dependence
of vectors, the two-sorted theory of vector-spaces with their
scalar-fields is companionable, and it is the union of a chain of
companionable theories, but the model-companions of the theories in
the chain are mutually inconsistent.  Finally, the union of a chain of
non-companionable theories may be companionable. 
\end{abstract}

\maketitle

A \textbf{theory} in a given signature is a set of sentences, in the
first-order logic of that signature, that is closed under logical
implication.  We shall consider chains $(T_m\colon m\in\upomega)$ of
theories: this means 
  \begin{equation}\label{eqn:no*}
    T_0\included T_1\included T_2\included\dotsb
  \end{equation}
The signature of $T_m$ will be $\sig_m$, 
so automatically $\sig_0\included\sig_1\included\sig_2\included\dotsb$%

In one
motivating example, $\sig_m$ is
$\{0,1,-,+,{}\cdot{},\partial_0,\dots,\partial_{m-1}\}$, the signature of
fields with $m$ additional singulary operation-symbols; and $T_m$ is
$\mDF$, the theory of fields (of any characteristic) with $m$ 
commuting derivations.  In this example, each $T_{m+1}$ is a
\textbf{conservative extension} of $T_m$, that is, $T_{m+1}\includes
T_m$ and every sentence in
$T_{m+1}$ of signature $\sig_m$ is already in $T_m$.  We establish
this by showing that every model of $T_m$ expands to a model of
$T_{m+1}$.  (This condition is sufficient, but not necessary
\cite[\S2.6, exer.~8, p.~66]{MR94e:03002}.)  If
$(K,\partial_0,\dots,\partial_{m-1})\models\mDF$, then
$(K,\partial_0,\dots,\partial_m)\models\mpDF$, where $\partial_m$ is
the $0$-derivation. 

The union of the theories $\mDF$ can be denoted by $\oDF$: it is the
theory of fields with $\upomega$-many commuting derivations.  Each of the
theories $\mDF$ has a \emph{model-companion,} called $\mDCF$
\cite{2007arXiv0708.2769P}; but we shall show (as
Theorem~\ref{thm:oDF} below) that $\oDF$ has no model-companion.  Let
us recall that a \textbf{model-companion} of a theory $T$ is a 
theory $T^*$ in the same signature such that
\begin{inparaenum}
\item
$T_{\forall}=T^*{}_{\forall}$, that is, every model of one of the
theories embeds in a model of the other, and
\item
$T^*$ is \textbf{model-complete,}
that is, $T^*\cup\diag{\str M}$ axiomatizes a complete theory for all
models $\str M$ of $T^*$.  
\end{inparaenum}
Here $\diag{\str M}$ is the quantifier-free theory of $\str M$
with parameters: equivalently, $\diag{\str M}$ is the theory of all
structures in which $\str M_M$ embeds.  (These notions, with
historical references, are reviewed further in
\cite{2007arXiv0708.2769P}.)  A theory has at most one
model-companion, by an argument with interwoven elementary chains. 

Let $\mDF_0$ be $\mDF$ with the additional requirement that the field
have characteristic $0$.  Then $\mDF_0$ has a model-companion, called
$\mDCF_0$ \cite{MR2001h:03066}.  We shall show (as
Theorem~\ref{thm:oDF0} below) that $\mDCF_0\included\mpDCF_0$.  It
will follow then that the union $\oDF_0$ of the $\mDF_0$ has a
model-companion, which is the union of the $\mDCF_0$.  This is by the
following general result, which has been observed also by Alice
Medvedev~\cite{Medvedev,Medvedev-preprint}.  Again, the theories $T_k$
are as in \eqref{eqn:no*} above. 

\begin{theorem}\label{thm:1}
Suppose each theory $T_k$ has a model-companion $T_k{}^*$, and
  \begin{equation}\label{eqn:*}
    T_0{}^*\included T_1{}^*\included T_2{}^*\included\dotsb
  \end{equation}
  Then the theory $\bigcup_{k\in\upomega}T_k$ has a model-companion,
  namely $\bigcup_{k\in\upomega}T_k{}^*$.
\end{theorem}

\begin{proof}
Write $U$ for $\bigcup_{k\in\upomega}T_k$, and $U^*$ for
$\bigcup_{k\in\upomega}T_k{}^*$. 
Suppose $\str A\models U$, and $\Gamma$ is a finite subset of
$U^*\cup\diag{\str A}$.  Then $\Gamma$ is a subset of
$T_k{}^*\cup\diag{\str A\restriction\sig_k}$ for some $k$ in
$\upomega$, and also $\str A\restriction\sig_k\models T_k$.  Since
$(T_k{}^*)_{\forall}\included T_k$, the structure $\str
A\restriction\sig_k$ must embed in a model of $T_k{}^*$; and this
model will be a model of $\Gamma$.  We conclude that $\Gamma$ is
consistent.  Therefore $U^*\cup\diag{\str A}$ is consistent.  Thus
$U^*{}_{\forall}\included U$.  By symmetry $U_{\forall}\included
U^*$. 

Similarly, if $\str B\models U^*$, then $T_k{}^*\cup\diag{\str
  B\restriction\sig_k}$ axiomatizes a complete theory in each case,
and therefore $U^*\cup\diag{\str B}$ is complete. 
\end{proof}

The foregoing proof does not require that the signatures $\sig_k$ form a chain,
but needs only that every finite subset of
$\bigcup_{k\in\upomega}\sig_k$ be included in some $\sig_k$.  This is
the setting for Medvedev's \cite[Prop.~2.4, p.~6]{Medvedev-preprint}, which then
has the same proof as the foregoing.  Also in Medvedev's setting, each $T_{k+1}{}^*$ is a conservative extension of $T_k{}^*$; but only
the weaker assumption $T_k{}^*\included T_{k+1}{}^*$ is needed in the
proof. 

Medvedev notes that many properties that the theories $T_k$ might have
are `local' and are therefore preserved in $\bigcup_{k\in\upomega}T_k$: examples are completeness, elimination of quantifiers, stability, and simplicity.  In her 
main application, $\sig_n$ is the signature of fields with singulary 
operation-symbols $\sigma_{m/n!}$, where $m\in\Z$; and $T_n$ is the
theory of fields on which the $\sigma_{m/n!}$ are automorphisms such that
\begin{equation*}
  \sigma_{k/n!}\circ\sigma_{m/n!}=\sigma_{(k+m)/n!}.
\end{equation*}
Then $T_n$ includes the theory $S_n$ of fields with
the single automorphism $\sigma_{1/n!}$.  
Using \cite[\S1]{MR2505433} (which is based on
\cite[ch.~5]{MR94e:03002}), we may observe at this point that
reduction of models of $T_n$
to models of $S_n$ is actually an equivalence of the categories
$\Mod{T_n}$ and $\Mod{S_n}$, whose objects are models of the
indicated theories, and whose morphisms are embeddings.
We thus have at hand a (rather simple) instance of the hypothesis of
the following theorem.  

\begin{theorem}\label{thm:bi}
  Suppose $(I,J)$ is a bi-interpretation of theories $S$ and $T$ such
  that $I$ is an equivalence of the categories $\Mod S$ and $\Mod T$.
  If $S$ has the model-companion $S^*$, and $S\included S^*$, then $T$
  also has a model-companion, which is the theory of those models
  $\str B$ of $T$ such that $J(\str B)\models S^*$.
\end{theorem}

\begin{proof}
  The class of  models $\str B$ of $T$ such that $J(\str B)\models
  S^*$ is elementary.  Let $T^*$ be its theory.  Then $T\included
  T^*$.  Suppose $\str B\models T$.
  Then $J(\str B)\models S$, so $J(\str B)$ embeds in a model $\str A$
  of $S^*$.  Consequently $I(J(\str B))$ embeds in $I(\str A)$.  Also
  $I(\str A)\models T^*$, since $\str A\cong J(I(\str A))$.  Since
  also $\str B\cong I(J(\str B))$, we conclude that $\str B$ embeds in
  a model of $T^*$.  Finally, $T^*$ is model-complete.  Indeed,
  suppose now $\str B$ and $\str C$ are models of $T^*$ such that
  $\str B\included\str C$.  An embedding of $J(\str B)$ in $J(\str C)$ is induced,
  and these structures are models of $S^*$, so the embedding is
  elementary.  Therefore the induced embedding of $I(J(\str B))$ in
  $I(J(\str C))$ is also elementary.  By the equivalence of the categories, $\str
  B\preccurlyeq\str C$.
\end{proof}

In the present situation, the theory $S_n$ has a model-companion
\cite{MR99c:03046,MR2000f:03109}; let us denote this by $\ACFA_n$.
By the theorem then, $T_n$ has a model-companion $T_n{}^*$, which is
axiomatized by $T_n\cup\ACFA_n$.  We have $\ACFA_n\included
T_{n+1}{}^*$ by \cite[1.12, Cor.~1, p.~3013]{MR2000f:03109}.  By
Theorem~\ref{thm:1} then, $\bigcup_{n\in\upomega}T_n$ has a
model-companion, which is the union of the $T_n{}^*$.  Medvedev calls
this union $\Q\ACFA$; she shows for example that it preserves the
simplicity of the $\ACFA_n$, as noted above, though it does not
preserve their supersimplicity.

The following is similar to the result that the theory of fields with
a derivation \emph{and} an automorphism (of the field-structure only)
has no model-companion \cite{MR2114160}.  The obstruction lies in positive
characteristics $p$, where all derivatives of elements with $p$-th
roots must be $0$.

\begin{theorem}\label{thm:oDF}
  The theory $\oDF$ has no model-companion.
\end{theorem}

\begin{proof}
We use that an $\forall\exists$ theory $T$ has a model-companion if
and only if the class of its \emph{existentially closed} models is
elementary, and in this case the model-companion is the theory of this
class \cite{MR0277372}.  (A model $\str A$ of $T$ is an
\textbf{existentially closed} model, provided that if $\str B\models
T$ and $\str A\included\str B$, then $\str A\preccurlyeq_1\str B$,
that is, all quantifier-free formulas over $A$ that are soluble in
$\str B$ are soluble in $\str A$.)  For each $n$ in $\upomega$, the
theory $\oDF$ has an
existentially closed model $\str A_n$, whose underlying field includes
$\mathbb F_p(\alpha)$, where $\alpha$ is transcendental; and in this model,
\begin{equation*}
  \partial_k\alpha=
  \begin{cases}
    1,&\text{ if }k=n,\\
0,&\text{ otherwise.}
  \end{cases}
\end{equation*}
Then $\alpha$ has no $p$-th root in $\str A_n$.  Therefore, in a
non-principal ultraproduct of the $\str A_n$, $\alpha$ has no
$p$-th root, although $\partial_n\alpha=0$ for all $n$ in $\upomega$, so that $\alpha$ does have a $p$-th root in some extension.
Thus the ultraproduct is not an existentially closed model of
$\oDF$.  Therefore the class of existentially closed models of $\oDF$ is not elementary.
\end{proof}

It follows then by Theorem~\ref{thm:1} that $\mDCF\nincluded\mpDCF$
for at least one $m$.  In fact this is so for all $m$, since
\begin{equation*}
  \mDCF\vdash p=0\lto\Forall
  x\Bigl(\bigwedge_{i<m}\partial_ix=0\lto\Exists yy^p=x\Bigr),
\end{equation*}
but $\mpDCF$ does not  entail this sentence, since
\begin{equation*}
  \mpDCF\vdash\Exists
x\bigl(\bigwedge_{i<m}\partial_ix=0\land\partial_mx\neq0\bigr).
\end{equation*}
However, this observation by
itself is not enough to establish the last theorem.  For, by the
results of \cite{MR2505433}, it is possible for each $T_k$ to have a
model-companion $T_k{}^*$, while $\bigcup_{k\in\upomega}T_k$ has a model-companion that
is not $\bigcup_{k\in\upomega}T_k{}^*$.  We may even require $T_{k+1}$ to be a conservative extension of $T_k$.

Indeed, if $k>0$, then in the notation of \cite{MR2505433}, $\VSp k$
is the theory of vector-spaces with their scalar-fields in the
signature $\{+,-,\bm0,\circ,0,1,*,P^k\}$, where $\circ$ is
multiplication of scalars, and $*$ is the action of the scalar-field
on the vector-space, and $P^k$ is $k$-ary linear dependence.  In
particular, $P^2$ may written also as $\parallel$.
Then $\VSp k$ has a model-companion, $\VSpst k$, which is the theory
of $k$-dimensional vector-spaces over algebraically
closed fields \cite[Thm~2.3]{MR2505433}.
Let
$\VSp{\upomega}=\bigcup_{1\leq k<\upomega}\VSp k$.  (This was
called $\VSp{\infty}$ in \cite{MR2505433}.)  This theory has the
model-companion $\VSpst{\upomega}$, which is the theory of
infinite-dimensional vector-spaces over algebraically closed
fields \cite[Thm~2.4]{MR2505433}.  In particular $\VSpst{\upomega}$ is not the union of the $\VSpst k$, because these are mutually inconsistent.  We now turn this into a result about chains:

\begin{theorem}\label{thm:vs}
If $1\leq n<\upomega$, let $T_n$ be the theory axiomatized by $\VSp 1\cup\dots\cup\VSp n$.
Then $T_n$ has a
  model-companion $T_n{}^*$, which is axiomatized by $T_n\cup\VSpst n$.  Also $T_{n+1}$ is a conservative extension of $T_n$.
  However, the model-companion $\VSpst{\upomega}$ of the union
  $\VSp{\upomega}$ of the
  chain $(T_n\colon 1\leq n<\upomega)$ is not the union of the $T_n{}^*$.
\end{theorem}

\begin{proof}
Every vector-space can be considered as a model of every $\VSp k$ and
hence of every $T_k$.  In particular, $T_{n+1}$ is a conservative
extension of $T_n$.  If the theories $T_n{}^*$ are as claimed, then
they are mutually inconsistent, and so $\VSpst{\upomega}$ is not their
union.  It remains to show that there are theories $T_n{}^*$ as
claimed.  We already know this when $n=1$.  For the other cases, if
$1\leq k<n$, we define the relations $P^k$ in models of $\VSp n$ of
dimension at least $n$. 

Let $\VSpm n$ the theory of such models: that is, $\VSpm n$ is
axiomatized by $\VSp n$ and the requirement that the space have
dimension at least $n$.  The relation $P^1$ is defined in models of
$\VSpm n$ (and indeed in models of $\VSp n$) by the quantifier-free
formula $\vec x=\vec0$.  If $n>2$, then there are existential formulas
that, in each model of $\VSpm n$, define the relation $\parallel$ and
its complement \cite[\S2, p.~431]{MR2505433}.  More generally, if
$1\leq k<n-1$, then, using existential formulas, we can define
$P^{k+1}$ and its complement in models of $T_k\cup\VSpm n$ or just
$\VSp k\cup\VSpm n$. Indeed, $\lnot P^{k+1}\vec x_0\cdots\vec x_k$ is
equivalent to $\Exists{(\vec x_{k+1},\dots,\vec x_{n-1})}\lnot P^n\vec
x_0\cdots\vec x_{n-1}$, and $P^{k+1}\vec x_0\cdots\vec x_k$ is
equivalent to 
  \begin{equation*}
\Exists{(\vec x_{k+1},\dots,\vec x_n)}\biggl(P^k\vec x_1\cdots\vec
x_k\lor\Bigl(\lnot P^n\vec x_1\cdots\vec
x_n\land\bigwedge_{j=k+1}^nP^n\vec x_0\cdots\vec x_{j-1}\vec
x_{j+1}\cdots\vec x_n\Bigr)\biggr). 
\end{equation*}
For, in a space of dimension at least $n$, if $(\vec a_0,\dots,\vec
a_k)$ is linearly dependent, but $(\vec a_1,\dots,\vec a_k)$ is not,
this means precisely that $(\vec a_1,\dots,\vec a_n)$ is independent
for some $(\vec a_{k+1},\dots,\vec a_n)$, but $\vec a_0$ is a
\emph{unique} linear combination of $(\vec a_1,\dots,\vec a_n)$, and
in fact of $(\vec a_1,\dots,\vec a_{j-1},\vec a_{j+1},\dots\vec a_n)$
whenever $k+1\leq j\leq n$, and (therefore) of $(\vec a_1,\dots,\vec
a_k)$. 

By \cite[Lem~1.1, 1.2]{MR2505433}, if $1\leq k<n-1$, we now have that
reduction from models of $T_{k+1}\cup\VSpm n$ to models of
$T_k\cup\VSpm n$ is an equivalence of the categories
$\Mod{T_{k+1}\cup\VSpm n}$ and $\Mod{T_k\cup\VSpm n}$.  Combining
these results for all $k$, we have that reduction from models of
$T_{n-1}\cup\VSpm n$ to models of $\VSpm n$ is an equivalence of the
categories $\Mod{T_{n-1}\cup\VSpm n}$ and $\Mod{\VSpm n}$.  Since
$\VSp n\included\VSpm n$ and every model of $\VSp n$ embeds in a model
of $\VSpm n$, the two theories have the
same model-companion, namely $\VSpst n$.  Similarly, $T_n$ and
$T_{n-1}\cup\VSpm n$ have the same model-companion; and by
Theorem~\ref{thm:bi}, this is axiomatized by $T_n\cup\VSpst n$. 
\end{proof}

A one-sorted version of the last theorem can be developed as follows.
Let $\VSpr n$ comprise the sentences of $\VSpm n$ having one-sorted
signature $\{\vec0,-,+,P^n\}$ of the sort of vectors alone. It is not
obvious that all models of $\VSpm n$ can be furnished with
scalar-fields to make them models of $\VSpr n$ again; but this will be
the case.  By \cite[Thm~1.1]{MR2505433}, it is the case when $n=2$:
reduction of models of $\VSpm2$ to models of $\VSpr2$ is an
equivalence of the categories $\Mod{\VSpm2}$ and $\Mod{\VSpr2}$.  This
reduction is therefore \textbf{conservative,} by the definition of
\cite[p.~426]{MR2505433}.  It is said further at
\cite[p.~431]{MR2505433} that reduction from $\VSpm n$ to $\VSpr n$ is
conservative when $n>2$; but the details are not spelled out.
However, the claim can be established as follows.  Immediately,
reduction from $\VSp2\cup\VSpm n$ to $\VSpr2\cup\VSpr n$ is
conservative.  In particular, models of the latter set of sentences
really are vector-spaces without their scalar-fields.  It is noted in
effect in the proof of Theorem \ref{thm:vs} that reduction from
$\VSp2\cup\VSpm n$ to $\VSpm n$ is conservative.  Furthermore, in
models of the latter theory, the defining of parallelism and its
complement is done with existential formulas \emph{in the signature of
  vectors alone.}  Therefore reduction from $\VSpr2\cup\VSpr n$ to
$\VSpr n$ is conservative.  We now have the following commutative
diagram of reduction-functors, three of them being conservative, that
is, being equivalences of categories. 
\begin{equation*}
\xymatrix{\Mod{\VSp2\cup\VSpm n}\ar[r]\ar[d]&\Mod{\VSpm n}\ar[d]\\
\Mod{\VSpr2\cup\VSpr n}\ar[r]&\Mod{\VSpr n}}
\end{equation*}
Therefore the remaining reduction, from $\VSpm n$ to $\VSpr n$, must
be conservative. 

Now there is a version of Theorem \ref{thm:vs} where $T_n$ is
axiomatized by $\VSpr2\cup\dots\cup\VSpr n$.  Indeed, by Theorem
\ref{thm:bi}, $T_n$ has a model-companion, which is the theory (in the
same signature) of $n$-dimensional vector-spaces over algebraically
closed fields; and the union of the $T_n$ has a model-companion, which
is the theory of infinite-dimensional vector-spaces over algebraically
closed fields; but this theory is not the union of the
model-companions of the $T_n$. 

The implication $\ref{ext}\Rightarrow\ref{red}$ in the following is
used implicitly at \cite[1.12, p.~3013]{MR2000f:03109} to establish
the result used above, that if $(K,\sigma)$ is a model of $\ACFA$,
then so is $(K,\sigma^m)$, assuming $m\geqslant 1$. 

\begin{theorem}\label{thm:cond}
Assuming as usual $T_0\included T_1$, where each $T_k$ has signature
$\sig_k$, we consider the following conditions. 
\begin{compactenum}[$A$.]
\item\label{ext}
For every model $\str A$ of $T_1$ and model $\str B$ of
  $T_0$ such that
  \begin{equation}\label{eqn:B}
  \str A\restriction\sig_0\included\str B,  
  \end{equation}
there is a model $\str C$ of $T_1$ such that
\begin{align}\label{eqn:ABC}
\str A&\included\str C,&
\str B&\included\str C\restriction\sig_0.
\end{align}
\item\label{red}
The reduct to $\sig_0$ of every existentially closed model of $T_1$ is an
existentially closed model of $T_0$.
\item\label{AP}
$T_0$ has the \emph{Amalgamation Property:} if one model embeds in two others, then those two in turn embed in a fourth model, compatibly with the original embeddings.
\item\label{ec}
$T_1$ is $\forall\exists$ (so that every model embeds in an existentially closed model).
\end{compactenum}
We have the two implications
\begin{align*}
\ref{ext}&\implies\ref{red},&
\ref{red}\And\ref{AP}\And\ref{ec}&\implies\ref{ext},
\end{align*}
but there is no implication among the four conditions that does not follow from these.  This is true, even if $T_1$ is required to be a conservative extension of $T_0$.
\end{theorem}

\begin{proof}
Suppose $\ref{ext}$ holds.
Let $\str A$ be an existentially closed model of $T_1$, and let $\str B$ be an arbitrary model of $T_0$ such that \eqref{eqn:B} holds.  By hypothesis, there is a model $\str C$ of $T_1$ such that \eqref{eqn:ABC} holds.  Then
 $\str
A\preccurlyeq_1\str C$, and therefore
$\str A\restriction\sig_0\preccurlyeq_1\str C\restriction\sig_0$, 
and \emph{a fortiori}
$\str A\restriction\sig_0\preccurlyeq_1\str B$.  Therefore $\str
A\restriction\sig_0$ must be an existentially closed model of $T_0$.  Thus $\ref{red}$ holds.

Suppose conversely $\ref{red}$ holds, along with $\ref{AP}$ and $\ref{ec}$.  Let $\str A\models T_1$ and $\str B\models T_0$ such that \eqref{eqn:B} holds.  We establish the consistency of
$T_1\cup\diag{\str A}\cup\diag{\str B}$.
It is enough to show the consistency of
\begin{equation}\label{eqn:T}
T_1\cup\diag{\str A}\cup\{\Exists{\vec x}\phi(\vec x)\},
\end{equation}
where $\phi$ is an arbitrary quantifier-free formula of $\sig_0(A)$ that is soluble in $\str B$.
By $\ref{ec}$, there is an existentially closed model $\str C$ of $T_1$ that extends $\str A$.  By $\ref{red}$ then, $\str C\restriction\sig_0$ is an existentially closed model of $T_0$ that extends $\str A\restriction\sig_0$.  By $\ref{AP}$, both $\str B$ and $\str C\restriction\sig_0$ embed over $\str A\restriction\sig_0$ in a model of $T_0$.  In particular, $\phi$ will be soluble in this model.  Therefore $\phi$ is already soluble in $\str C\restriction\sig_0$ itself.  Thus $\str C$ is a model of \eqref{eqn:T}.  Therefore $\ref{ext}$ holds.

The foregoing arguments eliminate the five possibilities marked $X$ on the table below, \begin{table}[ht]
\begin{equation*}
\begin{array}{r|*{16}{c}} &\text{\ref{1}}&X&\text{\ref{2}}&\text{\ref{3}}&\text{\ref{4}}&X&\text{\ref{5}}&\text{\ref{6}}&\text{\ref{7}}&X&\text{\ref{8}}&\text{\ref{9}}&\text{\ref{10}}&X&X&\text{\ref{11}}\\\hline
\ref{ext}&0&1&0&1&0&1&0& 1&0&1&0&  1&0&1&0& 1\\
\ref{red}&0&0&1&1&0&0&1& 1&0&0&1&  1&0&0&1& 1\\
\ref{AP}&0&0&0&0&1&1&1& 1&0&0&0&  0&1&1&1& 1\\
\ref{ec}&0&0&0&0&0&0&0& 0&1&1&1&  1&1&1&1& 1
\end{array}
\end{equation*}
\end{table}
where $0$ means false, and $1$, true.  We give examples of each of the
remaining cases, numbered according to the table.  In each example,
$T_0$ will be the reduct of $T_1$ to $\sig_0$.  We shall denote by
$\sig_{\mathrm f}$ the signature $\{+,{}\cdot{},-,0,1\}$ of fields;
and by $\sig_{\mathrm{vs}}$, the signature $\{+,-,\bm0,\circ,0,1,*\}$
of vector-spaces as two-sorted structures. 
\begin{asparaenum}[1.]
\item\label{1}
We first give an example in which none of the four lettered conditions
hold.  Let $\sig_0=\sig_{\mathrm f}\cup\{a,b\}$ and
$\sig_1=\sig_0\cup\{c\}$.  Let $T_1$ be the theory of fields of
characteristic $p$ with distinguished elements $a$, $b$, and $c$ such
that $\{a,c\}$ or $\{b,c\}$ is $p$-independent, and if $\{b,c\}$ is
$p$-independent, then so is $\{b,c,d\}$ for some $d$.  Then $T_0$ is
the theory of fields of characteristic $p$ in which, for some $c$,
$\{a,c\}$ or $\{b,c\}$ is $p$-independent, and if $\{b,c\}$ is
$p$-independent, then so is $\{b,c,d\}$ for some $d$.  The negations
of the four lettered conditions are established as follows.
Throughout, $a$, $b$, $c$, and $d$ will be algebraically independent
over $\F_p$. 
\begin{compactitem}
\item[$\lnot\ref{ext}$.]
We have
\begin{align*}
(\F_p(a,b^{1/p},c),a,b,c)&\models T_1,&
(\F_p(a,b^{1/p},c^{1/p}),a,b)&\models T_0,
\end{align*} 
but if $(\F_p(a,b^{1/p},c),a,b,c)$ is a substructure of a model $(K,a,b,c)$ of $T_1$, then $K$ cannot contain $c^{1/p}$.
\item[$\lnot\ref{red}$.]
$T_0$ has no existentially closed models, since an element of a model that is $p$-independent from $a$ or $b$ will always have a $p$-th root in some extension.  Similarly, no model of $T_1$ in which $\{a,c\}$ is not $p$-independent is existentially closed.
But $T_1$ does have existentially closed models, which are just the separably closed fields of characteristic $p$ with $p$-basis $\{a,c\}$ and with an additional element $b$. 
\item[$\lnot\ref{AP}$.]
$T_0$ does not have the Amalgamation Property, since $(\F_p(a,b^{1/p},c),a,b)$ and $(\F_p(a^{1/p},b,c,d),a,b)$ are models that do not embed in the same model over the common substructure $(\F_p(a,b,c),a,b)$, which is a model of $T_0$.
\item[$\lnot\ref{ec}$.]
$T_1$ is not $\forall\exists$, since, as we have already noted, models in which $\{a,c\}$ is not $p$-independent do not embed in existentially closed models.
\end{compactitem}
\item\label{2}
For an example of the column headed by $\ref{2}$ in the table, we let
$\sig_0$ and $\sig_1$ be as in $\ref{1}$; but now $T_1$ is the theory
of fields of characteristic $p$ with distinguished elements $a$, $b$,
and $c$ such that $\{a,c,d\}$ or $\{b,c,d\}$ is $p$-independent
for some $d$.  This ensures that $T_1$ has no
existentially closed models, so $\ref{red}$ holds vacuously; but the
other three conditions still fail. 
\item\label{3}
$T_0$ and $T_1$ are the same theory, so $\ref{ext}$ and $\ref{red}$
  hold trivially; and this theory is the theory of vector-spaces of
  dimension at least $2$, in the signature $\sig_{\mathrm{vs}}$, so
  the theory
  neither has the Amalgamation Property, nor is $\forall\exists$. 
\item\label{4}
$T_1$ is $\DF_p$ with the additional requirement that the field have
  $p$-dimension at least $2$; and $\sig_0=\sig_{\mathrm f}$, so $T_0$
  is the theory of fields of characteristic $p$ with $p$-dimension at
  least $2$.  The latter theory has the Amalgamation Property; but the
  other conditions fail.  Indeed, let $(\F_p(a,b),D)$ be the model of
  $T_1$ in which $Da=1$ and $Db=0$: then the field $\F_p(a,b)$ embeds
  in $\F_p(a^{1/p},b)$, which is a model of $T_0$, but $D$ does not
  extend to this field.  Also, $T_0$ has no existentially closed
  models; but $T_1$ does, and indeed it has a model-companion, namely
  $\DCF_p$.  Also $T_1$ is not $\forall\exists$, since $T_0$ is not:
  there is a chain of models of the latter, whose union is not a
  model, and we can make the structures in the chain into models of
  $T_1$ by adding the zero derivation. 
\item\label{5}
$\sig_0=\sig_{\mathrm f}$, and $\sig_1=\sig_0\cup\{a\}$.  $T_1$ is the
  theory of fields of characteristic $p$ with distinguished element
  $a$, which is $p$-independent from another element; so $T_0$ is (as
  in $\ref{4}$) the theory of fields of characteristic $p$ with
  $p$-dimension at least $2$.  Then we already have that $\ref{AP}$
  holds.  But $\ref{ext}$ fails: just let $\str A$ be $(\F_p(a,b),a)$,
  and let $\str B$ be $\F_p(a^{1/p},b)$.  Also $T_1$ has no
  existentially closed models, so $\ref{red}$ holds trivially, but
  $T_1$ is not $\forall\exists$.
\item\label{6}
$T_0$ and $T_1$ are the same, namely the theory of fields of characteristic $p$ of positive $p$-dimension, in the signature of fields, so this theory has the Amalgamation Property, but is not $\forall\exists$.
\item\label{7}
$\sig_0=\sig_{\mathrm{vs}}$, $\sig_1=\sig_0\cup\{\parallel,\vec a,\vec
  b\}$, and $T_1$ is axiomatized by $\VSp2\cup\{\vec a\nparallel\vec
  b\}$, so it is $\forall\exists$.  Then $T_0$ is the theory of
  vector-spaces of dimension at least $2$.  As in Theorem~\ref{thm:vs}
  above, $T_1$ has a model-companion, namely the theory of
  vector-spaces over algebraically closed fields with basis $\{\vec
  a,\vec b\}$.  But $T_0$ has no existentially closed models, since
  for all independent vectors $\vec a$ and $\vec b$ in some model, the
  equation
  \begin{equation}\label{eqn:7}
 x*\vec a+y*\vec b=\vec 0
  \end{equation}
  is always soluble in some
  extension.  Thus $\ref{red}$ fails.  Then $T_0$ also does not have
  the Amalgamation Property, since the solutions of \eqref{eqn:7}
  may satisfy $2x^2=y^2$ in one extension, but $3x^2=y^2$ in another.
  Similarly, $\ref{ext}$ fails, since the reduct to $\sig_0$ of a
  model of $T_1$ may embed in a model of $T_0$ in which $\vec a$ and
  $\vec b$ are parallel. 
\item\label{8}
$\sig_0=\sig_{\mathrm{vs}}\cup\{\parallel\}$, $\sig_1=\sig_0\cup\{\vec
  a,\vec b\}$, and $T_1$ is axiomatized by $\VSp2$ together with 
\begin{equation}\label{eqn:ab}
\Forall x\Forall y(x*\vec a+y*\vec b=\vec0\lto2x^2=y^2).
\end{equation}
Then $T_0$ is the theory of vector-spaces such that either the
dimension is at least $2$, or the scalar field contains $\surd2$.  As
in $\ref{7}$, $T_0$ does not have the Amalgamation Property.
The theory $T_1$ is $\forall\exists$.  It also has the model $(\Q*\vec
a\oplus\Q*\vec b,\vec a,\vec b)$, and $\Q*\vec a\oplus\Q*\vec b$
embeds in the model $\Q(\surd2,\surd3)*\vec a$ of $T_0$ when we let
$\vec b=\surd3*\vec a$; but then the latter space embeds in no space
in which $\vec a$ and $\vec b$ are as required by \eqref{eqn:ab}.  So
$\ref{ext}$ fails.  Finally, $T_1$ has a model-companion, axiomatized
by $\VSpst2$ together with 
\begin{equation*}
\Exists x\Exists y(x*\vec a+y*\vec b=\vec0\land 2x^2=y^2\land x\neq0);
\end{equation*}
and $T_0$ has a model-companion, which is just $\VSpst2$; so $\ref{red}$ holds.
\item\label{9}
$T_0$ and $T_1$ are both $\VSp1$.
\item\label{10}
$T_1=\DF_p$, and $T_0$ is the reduct to $\sig_{\mathrm f}$, namely field-theory in characteristic~$p$.
\item\label{11}
$T_0$ and $T_1$ are both field-theory. \qedhere
\end{asparaenum}
\end{proof}

Now let $\oDCF_0=\bigcup_{m\in\upomega}\mDCF_0$.  We obtain a positive
application of Theorem~\ref{thm:1}.

\begin{theorem}\label{thm:oDF0}
For all $m$ in $\upomega$,
\begin{equation*}
\mDCF_0\included\mpDCF_0.
\end{equation*}
Therefore $\oDF_0$ has a model-companion, which is $\oDCF_0$.  This theory admits
  full elimination of quantifiers, is complete, and is properly stable.  
\end{theorem}

\begin{proof}
Suppose $(L,\partial_0,\dots,\partial_{m-1})$ is a model of $\mDF_0$,
and $L$ has a subfield $K$ that is closed under the $\partial_i$
(where $i<m$), and there is also a derivation $\partial_m$ on $K$ such
that $(K,\partial_0\restriction K,\dots,\partial_{m-1}\restriction
K,\partial_m)$ is a model of $\mpDF_0$.  We shall include
$(L,\partial_0,\dots,\partial_{m-1})$ in another model of $\mDF_0$,
namely a model that expands to a model of $\mpDF_0$ that includes
$(K,\partial_0,\dots,\partial_m)$.  Thus condition \ref{ext} of
Theorem \ref{thm:cond} will hold, and therefore condition \ref{red}
will hold: this means $\mDCF_0\included\mpDCF_0$.  Since $m$ is
arbitrary, it will follow by Theorem~\ref{thm:1} that $\oDCF_0$ is the
model-companion of $\oDF_0$.  
  
If $K=L$, we are done.  Suppose $a\in L\setminus K$.  We shall define
a differential field $(K\langle
a\rangle,\tilde{\partial}_0,\dots,\tilde{\partial}_m)$,   
where $a\in K\langle a\rangle$, and for each $i$ in $m$,
\begin{equation}\label{eqn:rest}
\tilde{\partial}_i\restriction K\langle a\rangle\cap L=\partial_i\restriction K\langle a\rangle\cap L,
\end{equation}
and $\tilde{\partial}_m\restriction K=\partial_m$.
Then we shall be able to repeat the process, in case $L\nincluded
K\langle a\rangle$: we can work with an element of $L\setminus
K\langle a\rangle$ as we did with $a$.  Ultimately we shall obtain the
desired model of $\mpDF_0$ with reduct that includes
$(L,\partial_0,\dots,\partial_{m-1})$. 

Considering $\upomega^{m+1}$ as the set of $(m+1)$-tuples of natural
numbers, we shall have 
\begin{equation*}
K\langle a\rangle=K(a^{\sigma}\colon\sigma\in\upomega^{m+1}),
\end{equation*}
where
\begin{equation}\label{eqn:sigma}
a^{\sigma}=\tilde{\partial}_0{}^{\sigma(0)}\dotsm\tilde{\partial}_m{}^{\sigma(m)}a.
\end{equation}
In particular then, by \eqref{eqn:rest}, we must have
\begin{equation*}
\sigma(m)=0\implies a^{\sigma}=\partial_0{}^{\sigma(0)}\dotsm\partial_{m-1}{}^{\sigma(m-1)}a.
\end{equation*}
Using this rule, we make the definition
\begin{equation*}
K_1=K(a^{\sigma}\colon\sigma(m)=0).
\end{equation*}
We may assume that the derivations $\tilde{\partial}_i$ have been defined so far that
\begin{align}\label{eqn:m0K}
i<m&\implies\tilde{\partial}_i\restriction K_1=\partial_i\restriction K_1,&
\tilde{\partial}_m\restriction K&=\partial_m\restriction K.
\end{align}
Then \eqref{eqn:sigma} holds when $\sigma(m)<1$.

Now suppose that, for some positive $j$ in $\upomega$, we have been
able to define the field $K(a^{\sigma}\colon\sigma(m)<j)$, and for
each $i$ in $m$, we have been able to define $\tilde{\partial}_i$ as a
derivation on this field, and we have been able to define
$\tilde{\partial}_m$ as a derivation from
$K(a^{\sigma}\colon\sigma(m)<j-1)$ to
$K(a^{\sigma}\colon\sigma(m)<j)$, so that \eqref{eqn:m0K} holds, and
\eqref{eqn:sigma} holds when $\sigma(m)<j$.  We want to define the
$a^{\sigma}$ such that $\sigma(m)=j$, and we want to be able to extend
the derivations $\tilde{\partial}_i$ appropriately. 

If $i<m+1$, then, as in
\cite[\S4.1]{2007arXiv0708.2769P}, we let $\bm i$ 
denote the characteristic function of $\{i\}$ on $m+1$: that is, $\bm
i$ will be
the element of $\upomega^{m+1}$ that takes the value $1$ at $i$ and
$0$ elsewhere.
Considered as a product structure, $\upomega^{m+1}$
inherits from $\upomega$ the binary operations $-$ and $+$.  For each $i$ in $m+1$, we have a derivation $\tilde{\partial}_i$ from $K(a^{\sigma}\colon(\sigma+\bm i)(m)<j)$ to $K(a^{\sigma}\colon\sigma(m)<j)$ such that \eqref{eqn:m0K} holds, and also, if $\sigma(m)<j$, then
\begin{equation}\label{eqn:cond}
\sigma(i)>0\implies\tilde{\partial}_ia^{\sigma-\bm i}=a^{\sigma}.
\end{equation}
We now define the $a^{\sigma}$, where $\sigma(m)=j$, so that, first of all, we can extend $\tilde{\partial}_m$ so that \eqref{eqn:cond} holds when $\sigma(m)=j$ and $i=m$; but we must also ensure that \eqref{eqn:cond} can hold also when $\sigma(m)=j$ and $i<m$.  To do this, we shall have to make an inductive hypothesis, which is vacuously satisfied when $j=1$.  We shall also proceed recursively again.  More precisely, we shall refine the recursion that we are already engaged in.

We well-order the elements $\sigma$ of $\upomega^{m+1}$ by the linear ordering $\lord$ determined by
the left-lexicographic ordering of the $(m+1)$-tuples
\begin{equation*}
\bigl(\sigma(m),\sigma(0)+\dots+\sigma(m-1),
\sigma(0),\sigma(1),\dots,\sigma(m-2)\bigr).
\end{equation*}
Then $(\upomega^{m+1},\lord)$ has the order-type of the ordinal
$\upomega^2$.  This is a difference from the linear ordering defined
in \cite[\S4.1]{2007arXiv0708.2769P} and elsewhere.  However, for all $\sigma$ and $\tau$ in $\upomega^{m+1}$, and all $i$ in $m+1$, we still have
\begin{equation*}
  \sigma\lord\tau\implies\sigma+\bm i\lord\tau+\bm i.
\end{equation*}
We have assumed that, when $\tau=(0,\dots,0,j)$, we have the field $K(a^{\sigma}\colon\sigma\lord\tau)$, together with, for each $i$ in $m+1$, a derivation $\tilde{\partial}_i$ from $K(a^{\xi}\colon\xi+\bm
i\lord \tau)$ to $K(a^{\xi}\colon\xi\lord \tau)$ such that \eqref{eqn:m0K} holds, and also, if $\sigma\lord\tau$, then
\eqref{eqn:cond} holds.  We have noted that we \emph{can} have all of this when $\tau=(0,\dots,0,1)$.  Suppose we have all of this for \emph{some} $\tau$ in $\upomega^{m+1}$ such that $(0,\dots,0,1)\lordeq\tau$, that is, $\tau(m)>0$.  We want to define the extension $K(a^{\sigma}\colon\sigma\lordeq\tau)$ of $K(a^{\sigma}\colon\sigma\lord\tau)$ so that we can extend the $\tilde{\partial}_i$ appropriately.  For defining $a^{\tau}$, there are two
cases to consider.  We use the rules for derivations gathered, for example, in \cite[Fact 1.1]{MR2114160}.
\begin{compactenum}[1.]
\item
If $a^{\tau-\bm m}$ is algebraic over
$K(a^{\xi}\colon\xi\lord \tau-\bm m)$, then the derivative $\tilde{\partial}_ma^{\tau-\bm
  m}$ is determined as an element of $K(a^{\xi}\colon\xi\lord \tau)$;
we let $a^{\tau}$ be this element.
  \item
If $a^{\tau-\bm m}$ is not algebraic over
$K(a^{\xi}\colon\xi\lord \tau-\bm m)$, then we let $a^{\tau}$ be
transcendental over $L(a^{\xi}\colon\xi\lord \tau)$.  We are then free
to define $\tilde{\partial}_ma^{\tau-\bm m}$ as $a^{\tau}$.  (We require $a^{\tau}$ to be transcendental over $L(a^{\xi}\colon\xi\lord \tau)$, and not just over $K(a^{\xi}\colon\xi\lord \tau)$, so that we can establish \eqref{eqn:rest} later.)
\end{compactenum}
We now check that, when $i<m$ and $\tau(i)>0$, we can define
$\tilde{\partial}_ia^{\tau-\bm i}$ as $a^{\tau}$.  Here we make the inductive hypothesis mentioned above, namely that the foregoing two-part definition of $a^{\tau}$ was already used to define $a^{\tau-\bm i}$.  Again we consider
two cases.
\begin{compactenum}[1.]
  \item
Suppose $a^{\tau-\bm i}$ is algebraic over $K(a^{\xi}\colon\xi\lord \tau-\bm
i)$.  Then $\tilde{\partial}_ia^{\tau-\bm i}$ is determined as an
element of $K(a^{\xi}\colon\xi\lord \tau)$.  Thus the value of the bracket
$[\tilde{\partial}_i,\tilde{\partial}_m]$ at $a^{\tau-\bm i-\bm m}$ is
determined: indeed, we have
\begin{equation*}
[\tilde{\partial}_i,\tilde{\partial}_m]a^{\tau-\bm i-\bm m}
=\tilde{\partial}_i\tilde{\partial}_ma^{\tau-\bm i-\bm m} -\tilde{\partial}_m\tilde{\partial}_ia^{\tau-\bm i-\bm m}
=\tilde{\partial}_ia^{\tau-\bm i}-a^{\tau}.
\end{equation*}
By inductive hypothesis,
since $a^{\tau-\bm i}$ is algebraic over $K(a^{\xi}\colon\xi\lord \tau-\bm
i)$, also $a^{\tau-\bm i-\bm m}$ must be algebraic over
$K(a^{\xi}\colon\xi\lord \tau-\bm i-\bm m)$.  Since the bracket is $0$ on
this field, it must be $0$ at $a^{\tau-\bm i-\bm m}$ as well
\cite[Lem.~4.2]{2007arXiv0708.2769P}. 
\item
If $a^{\tau-\bm i}$ is transcendental over
$K(a^{\xi}\colon\xi\lord \tau-\bm i)$, then, since we are given $\tilde{\partial}_i$ as
a derivation whose domain is this field, we are free to define
$\tilde{\partial}_ia^{\tau-\bm i}$ as $a^{\tau}$. 
\end{compactenum}
Thus we have obtained $K(a^{\xi}\colon\xi\lordeq\tau)$ as desired.  By induction, we obtain the differential field $(K(a^{\sigma}\colon\sigma\in\upomega^{m+1}),\tilde{\partial}_0,\dots,\tilde{\partial}_m)$ such that \eqref{eqn:sigma} and \eqref{eqn:m0K} hold.

It remains to check that \eqref{eqn:rest} holds.  It is enough to show
\begin{equation}\label{eqn:cap}
K\langle a\rangle\cap L\included K_1.
\end{equation}
(We have the reverse inclusion.)
Suppose $\tau\in\upomega^{m+1}$ and $\tau(m)>0$.  By the definition of $a^{\tau}$,
\begin{gather}\label{eqn:in}
	a^{\tau}\in K(a^{\sigma}\colon\sigma\lord\tau)\alg\implies a^{\tau}\in K(a^{\sigma}\colon\sigma\lord\tau),\\\label{eqn:notin}
	a^{\tau}\notin K(a^{\sigma}\colon\sigma\lord\tau)\alg\implies a^{\tau}\notin L(a^{\sigma}\colon\sigma\lord\tau)\alg.
\end{gather}
Suppose $b\in K\langle a\rangle\cap L$.  Since $b\in K\langle
a\rangle$, we have, for some $\tau$ in $\upomega^{m+1}$, that $b$ is a
rational function over $K_1$ of those
$a^{\sigma}$ such that $\bm m\lordeq\sigma\lordeq\tau$.  But then, by
\eqref{eqn:in}, we do not need any $a^{\sigma}$ that is algebraic over
$K(a^{\xi}\colon\xi\lord\sigma)$, since it actually belongs to this
field.  When we throw out all such $a^{\sigma}$, then, by
\eqref{eqn:notin}, those that remain are algebraically independent
over $L$.  Thus we have 
\begin{equation*}
b\in K_1(a^{\sigma_0},\dots,a^{\sigma_{n-1}})\cap L
\end{equation*}
for some $\sigma_j$ in $\upomega^{m+1}$ such that
$(a^{\sigma_0},\dots,a^{\sigma_{n-1}})$ is algebraically independent
over $L$.  Therefore we may assume $n=0$, and $b\in
K_1$.  Thus \eqref{eqn:cap} holds, and we
have the differential field $(K\langle
a\rangle,\tilde{\partial}_0,\dots,\tilde{\partial}_m)$ fully as
desired. 

We have to be able to repeat this contruction, in case $L\nincluded
K\langle a\rangle$.  If $b\in L\setminus K\langle a\rangle$, we have
to be able to construct $K\langle a,b\rangle$, and so on.  Let
$L\langle a\rangle$ be the compositum of $K\langle a\rangle$ and $L$.
Since $\mDF_0$ has the Amalgamation Property, we can extend the
$\tilde{\partial}_i$, where $i<m$, to commutating derivations on the field
$L\langle a\rangle$ that extend the original $\partial_i$ on $L$.
Thus we have a model $(L\langle
a\rangle,\tilde{\partial}_0,\dots,\tilde{\partial}_{m-1})$ of $\mDF_0$
and a model $(K\langle a\rangle,\tilde{\partial}_0\restriction
K\langle a\rangle,\dots,\tilde{\partial}_{m-1}\restriction K\langle
a\rangle,\tilde{\partial}_m)$ of $\mpDF_0$ that include, respectively,
the models that we started with.  Now we can continue as before,
ultimately extending the domain of $\tilde{\partial}_m$ to include all
of $L$.  At
limit stages of this process, we take unions, which is no problem,
since $\mDF_0$ and $\mpDF_0$ are $\forall\exists$. 

Therefore $\oDF_0$ has the model-companion $\oDCF_0$.  Since the
$\mDCF_0$ have the properties of quantifier-elimination, completeness,
and stability \cite{MR2001h:03066}, the observations of Medvedev noted
earlier allow us to conclude that $\oDCF_0$ also has these properties.
Although each $\mDCF_0$ is actually $\upomega$-stable, $\oDCF_0$ is
not even superstable, since if $A$ is a set of constants (in the sense
that all of their derivatives are $0$), then as $\sigma$ ranges over
$A^{\upomega}$, the sets $\{\partial_mx=\sigma(m)\colon
m\in\upomega\}$ belong to distinct complete types. 
\end{proof}

In the foregoing proof, we cannot use Condition $\ref{ext}$ of
Theorem~\ref{thm:cond} in the stronger form in which the structure
$\str C$ is required to be a mere \emph{expansion} to $\sig_1$ of
$\str B$: 

\begin{theorem}
If $m>0$, there is a model $\str K$ of $\mpDF_0$ with a reduct that is
included in a model $\str L$ of $\mDF_0$, while $\str L$ does not
expand to a model of $\mpDF_0$ that includes $\str K$. 
\end{theorem}

\begin{proof}
We generalize the example of \cite{MR96g:35006} repeated in
\cite[Ex.~1.2, p.~927]{MR2000487}.
Suppose $K$ is a pure transcendental extension
$\Q(a^{\sigma}\colon\sigma\in\upomega^{m+1})$ of $\Q$.  We make this into a model of $\mpDF_0$ by requiring $\partial_ia^{\sigma}=a^{\sigma+\bm i}$ in each case.  Let $L$ be the pure transcendental extension $K(b^{\tau}\colon\tau\in\upomega^{m-1})$ of
$K$.  We make this into a model of $\mDF_0$ by extending the $\partial_i$ so that, if $i<m-1$, we have $\partial_ib^{\tau}=b^{\tau+\bm i}$, while
$\partial_{m-1}b^{\tau}$ is the element $a^{(\tau,0,0)}$ of $K$.  Note that indeed if $i<m-1$, then
\begin{equation*}
[\partial_i,\partial_{m-1}]b^{\tau}=\partial_ia^{(\tau,0,0)}-\partial_{m-1}b^{\tau+\bm i}=0.
\end{equation*}
Suppose, if possible, $\partial_m$ extends to $L$ as well so as to commute with the other $\partial_i$.  Then for any $\tau$ in $\upomega^{m-1}$ we have $\partial_mb^{\tau}=f(b^{\xi}\colon\xi\in\upomega^{m-1})$ for some polynomial $f$ over $K$.  But then, writing $\partial_{\eta}f$ for the derivative of $f$ with respect to the variable indexed by $\eta$, we have, as by \cite[Fact 1.1(0)]{MR2114160},
\begin{align*}
a^{(\tau,0,1)}
&=\partial_m\partial_{m-1}b^{\tau}\\
&=\partial_{m-1}\partial_mb^{\tau}\\
&=\partial_{m-1}(f(b^{\xi}\colon\xi\in\upomega^{m-1}))\\
&=\sum_{\eta\in\upomega^{m-1}}\partial_{\eta}f(b^{\xi}\colon\xi\in\upomega^{m-1})\cdot a^{(\eta,0,0)}+f^{\partial_{m-1}}(b^{\xi}\colon\xi\in\upomega^{m-1}),
\end{align*}
where the sum has only finitely many nonzero terms.  The polynomial expression $f^{\partial_{m-1}}(b^{\xi}\colon\xi\in\upomega^{m-1})$ cannot have $a^{(\tau,0,1)}$ as a constant term, since this is not $\partial_{m-1}x$ for any $x$ in $K$.  Thus we have obtained an algebraic relation among the $b^{\sigma}$ and $a^{\tau}$; but there can be no such relation.
\end{proof}

Finally, the union of a chain of non-companionable theories may be companionable:

\begin{theorem}
In the signature $\{f\}\cup\{c_k\colon k\in\upomega\}$, where $f$ is a singulary operation-symbol and the $c_k$ are constant-symbols, let $T_0$ be axiomatized by the sentences
\begin{equation*}
\Forall x\Forall y(fx=fy\lto x=y)
\end{equation*}
and, for each $k$ in $\upomega$,
\begin{align*}
\Forall x(f^{k+1}x&\neq x),&
\Forall x(fx=c_k&\lto x=c_{k+1}),&
fc_{k+2}=c_{k+1}&\lto fc_{k+1}=c_k.
\end{align*}
For each $n$ in $\upomega$, let $T_{n+1}$ be axiomatized by
\begin{equation*}
  T_n\cup\{fc_{n+1}=c_n\}. 
\end{equation*}
Then
\begin{compactenum}
\item
each $T_n$ is universally axiomatized, and \emph{a fortiori} $\forall\exists$, so it does have existentially closed models;
\item
each $T_n$ has the Amalgamation Property;
 \item
 every existentially closed model of $T_{n+1}$ is an existentially closed model of~$T_n$;
\item
 no $T_n$ is companionable;
 \item
 $\bigcup_{n\in\upomega}T_n$ is companionable.
 \end{compactenum}
\end{theorem}

\begin{proof}
Let $\str A_m$ be the model of $T_0$ with universe $\upomega\times\upomega$ such that
\begin{align*}
f^{\str A_m}(k,\ell)&=(k,\ell+1),&
c_k{}^{\str A_m}&=
\begin{cases}
(k-m,0),&\text{ if }k>m,\\
(0,m-k),&\text{ if }k\leq m.	
\end{cases}
\end{align*}
Let $\str A_{\upomega}$ be the model of $T_0$ with universe $\Z$ such that
\begin{align*}
f^{\str A_{\upomega}}k&=k+1,&
c_k{}^{\str A_{\upomega}}&=-k.
\end{align*}
Then $\str A_m$ is a model of each $T_k$ such that $k\leq m$; and $\str A_{\upomega}$ is a model of each $T_k$.  Moreover, each model of $T_k$ consists of a copy of some $\str A_{\beta}$ such that $k\leq\beta\leq\upomega$, along with some (or no) disjoint copies of $\upomega$ and $\Z$ in which $f$ is interpreted as $x\mapsto x+1$.  Conversely, every structure of this form is a model of $T_k$.  The $\beta$ such that $\str A_{\beta}$ embeds in a given model of $T_k$ is uniquely determined by that model.  Consequently $T_k$ has the Amalgamation Property.
Also, a model of $T_k$ is an existentially closed model if and only if includes no copies of $\upomega$ (outside the embedded $\str A_{\beta}$):  This establishes that every existentially closed model of $T_{k+1}$ is an existentially closed model of $T_k$.

The existentially closed models of $T_k$ are those models that omit the type $\{\Forall yfy\neq x\}\cup\{x\neq c_j\colon j\in\upomega\}$.  In particular, $\str A_m$ is an existentially closed model of $T_k$, if $k\leq m$; but $\str A_m$ is elementarily equivalent to a structure that realizes the given type.  Thus $T_k$ is not companionable.

Finally, the model-companion of $\bigcup_{k\in\upomega}T_k$ is axiomatized by this theory, together with $\Forall x\Exists yfy=x$.
\end{proof}


\providecommand{\bysame}{\leavevmode\hbox to3em{\hrulefill}\thinspace}
\providecommand{\MR}{\relax\ifhmode\unskip\space\fi MR }
\providecommand{\MRhref}[2]{%
  \href{http://www.ams.org/mathscinet-getitem?mr=#1}{#2}
}
\providecommand{\href}[2]{#2}

  \end{document}